\documentclass[reqno,a4paper]{amsart}
\usepackage[utf8]{inputenc}
\usepackage[T1]{fontenc}
\usepackage{amssymb,amsthm,amsmath,mathtools}
\usepackage[hypertexnames=false,hidelinks]{hyperref}
\usepackage{xcolor}
\usepackage[all]{xy}

\usepackage{mathrsfs}

\usepackage{multirow}
\usepackage{booktabs}

\usepackage{color}

\newdir{>}{{}*:(1,-.2)@^{>}*:(1,+.2)@_{>}}
\newdir{<}{{}*:(1,+.2)@^{<}*:(1,-.2)@_{<}}

\theoremstyle{plain}
\newtheorem{theorem}[subsection]{Theorem}

\newtheorem{corollary}[subsection]{Corollary}

\theoremstyle{remark}
\newtheorem{definition}[subsection]{Definition}

\numberwithin{equation}{section}

\newcommand{\CC}{\mathbb{C}}
\newcommand{\C}{\mathcal{C}}
\newcommand{\M}{\mathcal{M}}
\newcommand{\N}{\mathcal{N}}
\newcommand{\git}{\mathord{/\mkern-6mu/}}
\newcommand{\CCC}{\mathbb{C}[\mathcal{C}_{42}]}

\DeclareMathOperator{\Tr}{Tr}

\DeclareMathOperator{\GL}{GL}
\DeclareMathOperator{\Com}{Com}

\begin{document}
	
	\title[Poisson structure and invariants of matrices]{Noncommutative Poisson structure and invariants of matrices}
	
	\author{F.~Eshmatov}
	\author{X.~García-Martínez}
	\author{R.~Turdibaev}
	
	\email{f.eshmatov@newuu.uz}
	\email{xabier.garcia.martinez@uvigo.gal}
	\email{rustam.turdibaev@usc.es}

	\address[Farkhod Eshmatov]{New Uzbekistan University, 100007 Tashkent, Uzbekistan --- Central Asian University, National Park Street 264, 111221, Tashkent, Uzbekistan}
	\address[Xabier Garc\'ia-Mart\'inez]{CITMAga \& Universidade de Vigo, Departamento de Matem\'aticas, Esc.\ Sup.\ de Enx.\ Inform\'atica, Campus de Ourense, E--32004 Ourense, Spain}
	\address[Rustam Turdibaev]{CITMAga \& Universidade de Santiago de Compostela, Departamento de Matemáticas, R\'ua Lope Gómez de Marzoa, s/n, 15782 Santiago de Compostela, Spain}

	\thanks{This work was supported by Agencia Estatal de Investigaci\'on de Espa\~{n}a (Spain, European FEDER support included), grants PID2020-115155GB-I00 and PID2021-127075NA-I00, and by Xunta de Galicia through the Competitive Reference Groups (GRC), ED431C 2023/31.}

	\begin{abstract}
		We introduce a novel approach that employs techniques from noncommutative Poisson geometry to comprehend the algebra of invariants of two~$n\times n$ matrices. We entirely solve the open problem of computing the algebra of invariants of two~$4 \times 4$ matrices. As an application, we derive the complete description of the invariant commuting variety of $4 \times 4$ matrices and the fourth Calogero-Moser space.
	\end{abstract}
	
	
	\maketitle
	
	%%%%%%%%%%%%%%%%%%%%%%%%%%%%%%%%%%%%%%%%%%%%%%%%%%%%%%%%%%%%%%%%%%%%%%%%%%%%%%%%%%%%%%%%%%%%%%%%%%%%%%%%%%%%%%%%%%%%%%%%%%%

	\section{Introduction}
	
	In the realm of algebraic geometry, there exists a categorical equivalence between affine schemes and commutative algebras. Many geometric structures on affine schemes can be algebraically described, and conversely. However, this correspondence is not applicable to associative algebras.
	
	Kontsevich and Rosenberg proposed a heuristic principle to explore noncommutative geometry on associative algebras, which may not necessarily be commutative~\cite{KR}. This principle can be broadly summarized: for an associative algebra $A$ over $\mathbb{C}$, any noncommutative geometric structure—such as noncommutative Poisson, noncommutative symplectic, etc.—should naturally induce its classical counterpart on the affine variety $\mathrm{Rep}_n(A)$ classifying all representations of $A$ on $\mathbb{C}^n$, 
	for every $n \in \mathbb{N}$.
	
	Crawley-Boevey introduced an \textit{$H_0$-Poisson structure} on any associative algebra~\cite{CB}. Essentially, it is a bilinear bracket on $A$ which, in turn, induces a Lie bracket on the commutator quotient space $A_{\natural}\coloneq A/[A,A]$. This structure aligns with the~KR principle.
	More precisely, let $\mathbb{C}[\mathrm{Rep}_n(A)]$ be the coordinate ring of~$\mathrm{Rep}_n(A)$. 
	The natural trace map 
	\begin{equation}
		\label{Trn}
		\mathrm{Tr}_n(A)\colon A_{\natural} \to \mathbb{C}[\mathrm{Rep}_n(A)] \, , \quad \bar{a} \to [\rho \mapsto \mathrm{Tr}\, \rho(a)] 
	\end{equation}
	sends the elements of $A_{\natural}$ to functions on $\mathrm{Rep}_n(A)$. The image of $\mathrm{Tr}_n(A)$ lies in
	the subalgebra $\mathbb{C}[\mathrm{Rep}_n(A)]^{\mathrm{GL_n}}$ which represents the coordinate ring of the character variety $\mathrm{Rep}_n(A)\git \mathrm{GL_n}$. Then the main result of~\cite{CB} states that if $A$ possesses a $H_0$-Poisson structure then there is a unique Poisson bracket on $\mathbb{C}[\mathrm{Rep}_n(A)]^{\mathrm{GL_n}}$ so that the trace map~\eqref{Trn} is a Lie algebra homomorphism. Equivalently, the algebra homomorphism induced from~\eqref{Trn}
	\begin{equation}
		\label{symtr}
		\mathrm{Sym \, {Tr}}_n(A)\colon \mathrm{Sym} (A_{\natural}) \to 
		\mathbb{C}[\mathrm{Rep}_n(A)]^{\mathrm{GL_n}}
	\end{equation}
	is a Poisson homomorphism.

	Let $\mathcal{M}_{n}$ be the vector space of $n\times n$ matrices over the field $\mathbb{C}$. Consider a positive integer $d$ and the action of the general linear group $\mathrm{GL}_n(\mathbb{C})$ on the direct product $\mathcal{M}_n^d$ of $d$ copies of $\mathcal{M}_n$ by simultaneous conjugation. This action sends a $d$-tuple $(X_1,\dots ,X_d)$ to another $d$-tuple $(gX_1g^{-1},\dots,gX_dg^{-1})$ for any $g \in \mathrm{GL}_n(\mathbb{C})$. Consequently, this action induces an action of $\mathrm{GL}_n$ on the algebra $\mathbb{C}[\mathcal{M}_n^d]$, representing polynomial functions on $\mathcal{M}_n^d$. One of the fundamental questions in invariant theory revolves around describing the generators and relations of the algebra~$C_{nd} \coloneq \mathbb{C} [\mathcal{M}_n^d]^{\mathrm{GL_n}}$. A well-known result by Procesi~\cite{Pr} and Razmyslov~\cite{Ra} states that $C_{nd}$ is generated by trace functions and the algebraic relations between them are derived from the Cayley-Hamilton theorem. Let us provide the following, more algebraic, interpretation of this result. Consider the free algebra~$R \coloneq \mathbb{C}\langle x_1,...,x_d \rangle$ on $d$ generators. Then each point $X \coloneq (X_1,...,X_d) \in \mathcal{M}_n^d$ can be uniquely determined by the representation
	$$ \rho_X \colon R \to \mathcal{M}_n\, , \quad x_i \mapsto X_i \, $$
	and two representations $\rho_X$ and $\rho_Y$ are isomorphic if and only if $X$ and $Y$ lie in the same $\mathrm{GL}_n$-orbit. Consequently, we have the equivalence 
	$$C_{nd} \cong \mathbb{C}[\mathrm{Rep}_n(R)]^{\mathrm{GL_n}}\, , $$ 
	and the Procesi-Razmyslov result can be interpreted as the map in~\eqref{symtr} being a surjective homomorphism. Thus, if $R$ has an $H_0$-Poisson structure, the map in~\eqref{symtr} becomes a surjective Poisson homomorphism. Consequently, describing the algebra~$C_{nd}$ is equivalent to providing a description of the kernel of $\mathrm{Sym \, {Tr}}_n(R)$,  which forms a Poisson ideal of $\mathrm{Sym} (R_{\natural})$.
	
	Let us discuss the $H_0$-Poisson structure proposed by Kontsevich on the free algebra with an even number of generators and the corresponding Lie algebra homomorphism. To this end, we utilize the quiver realization of a free algebra.
	Let~$Q=(I,H)$ be a finite quiver, where $I$ is the set of vertices and $H$ is the set of edges.
	For $\alpha \in \mathbb{N}^I$, we define the space of representations of $Q$ of dimension vector~$\alpha$ to be the vector space of matrices 
	$$ \mathrm{Rep}(Q, \alpha)=\bigoplus_{h\in H} \mathrm{Mat}(\alpha_{h_0}\times \alpha_{h_1},\mathbb{C}) \, ,$$
	where $h_0$ and $h_1$ are the head and tail of an edge $h$. The group $G(\alpha) = \prod \GL_{\alpha_i}(\mathbb{C})$ acts on $\mathrm{Rep}(Q, \alpha)$ by conjugation, and the quotient space $\mathrm{Rep}(Q, \alpha)/G(\alpha)$ represents isomorphism classes of such representations. Recall that the \textit{double} quiver~$\overline{Q}$ is defined by adding a reverse edge $a^{\ast}$ for each $a \in H$. Then there is a natural identification of $\mathrm{Rep}(\overline{Q}, \alpha)$ with the cotangent bundle $T^*\mathrm{Rep}(Q, \alpha)$. 
	The latter has a canonical symplectic structure, which via the symplectic reduction construction~\cite{MW}, induces a symplectic structure on $T^*\mathrm{Rep}(Q, \alpha)\git G(\alpha)$. This, in turn, induces a Poisson bracket on the algebra $\mathbb{C}[\mathrm{Rep}(\overline{Q}, \alpha)]^{G(\alpha)}$.
	
	In~\cite{Ko}, Kontsevich defined an infinite-dimensional Lie algebra structure on the vector space $\mathbb{C} \overline{Q}_{\natural}$, where  $\mathbb{C} \overline{Q}$ is the path algebra of $ \overline{Q}$. He showed that the map
	\begin{equation}
		\label{symtrQ}
		\mathrm{Tr}_{\alpha}(Q)\colon \mathbb{C} \overline{Q}_{\natural} \to 
		\mathbb{C}[\mathrm{Rep}(\overline{Q}, \alpha)]^{G(\alpha)}
	\end{equation}
	is a Lie algebra homomorphism for all $\alpha \in \mathbb{N}^{I}$. Let $Q$ be the quiver with one vertex and $k$ edges-loops. Then, the path algebra $\mathbb{C} \overline{Q}$ is isomorphic to the free algebra~$\mathbb{C}\langle x_1,\dots,x_{k},y_1,\dots,y_k \rangle$. The vector space $V=\mathrm{span}(x_1,\dots,x_{k},y_1,\dots,y_k)$ has a canonical symplectic form~$\omega$: 
	\[
	\omega(x_i,x_j)=\omega(y_i,y_j)=0 \quad 
	\textrm{and} \quad \omega(x_i,y_j)=\delta_{ij}.
	\]
	Then, the Kontsevich Lie bracket on $\mathbb{C} \overline{Q}_{\natural}$ is defined as follows:
	\begin{equation}	\label{Kont}
		\{u_1\cdots u_p, v_1\cdots v_q\}=\sum\limits_{\substack {
				1\leq i \leq p \\ 	  1\leq j \leq q	}}\omega(u_i,v_j)u_{i+1}\cdots u_pu_1\cdots u_{i-1}v_{j+1}\cdots v_q v_1\cdots v_{j-1},
	\end{equation}
	where the elements $u_1,\dots,u_p,v_1,\dots, v_q \in \{ x_1,...,x_k,y_1,...,y_k\}$. As mentioned above,  the map~\eqref{symtrQ} allows to compute the Poisson bracket on $C_{nd}$ for $d=2k$. For instance, the explicit Poisson bracket on $C_{32}$ was described this way by Normatov and Turdibaev~\cite{NT-JAA}. 
	
	The problem of identifying a minimal generating set of~$C_{nd}$ and finding the exact relations among these generators still remains  unsolved. Several noteworthy results in this area have been attained thus far and explicit descriptions were found for only $C_{2m}$ for all positive integers $m$, $C_{32}$ and $C_{33}$.  The main method utilized is based on the work of Abeasis and Pittaluga~\cite{AP}, using the representation theory of symmetric and general linear groups. They employed this approach in a manner reminiscent of its application in the theory of PI-algebras. Their method finds the minimal set of generators as a result of a decomposition of a certain module into irreducible components. 
	
	The next goal was to give the description of $C_{42}$. Drensky and Sadikova~\cite{DS} discovered the existence of a minimal generating set whose span forms a semisimple graded $\GL_2$-module, and they studied its structure. As for the defining relations, it seems that the question remained largely open. To address this, Drensky and La~Scala~\cite{DLaS} have devised algorithms based on the representation theory of $\GL_2$ and conducted calculations using standard Maple functions to determine all defining relations occurring in degrees 12, 13, and 14. 
	
	The present paper provides a comprehensive solution to the problem of identifying every defining relation with respect to the minimal system of generators of~$C_{42}$ in Drensky and Sadikova~\cite{DS}. Unlike previous approaches which were confined to generating relations within the same degree, our method coming from the noncommutative geometry offers the advantage of a progression from lower to higher degree relations through the Poisson algebra structure. Remarkably, the highest degree at which these defining relations appear is $20$. Nevertheless, our approach expedites the procedure, enabling the combinatorial computation to finalize at degree 16 with a partial analysis of bidegree (8, 8) being sufficient.  We demonstrate that by utilizing the Poisson algebra structure, a mere 8 relations suffice to build the associative ideal of defining relations, which is generated by 105 polynomials. The insights into how the Poisson structure is effectively harnessed to enhance our computation are given in Subsection~\ref{ss:refined}.

	With all the defining relations of $C_{42}$ at our disposal, we are confronted with the task of understanding them. One perspective in invariant theory suggests the following useful way of representing an invariant by its Hironaka decomposition. It is well-known, that $C_{nd}$ is Cohen-Macaulay(see~\cite{HR} and~\cite{VDB}). In other words, there exists a set of polynomials, called a homogeneous system of parameters or primary invariants $p_1,\dots, p_k$, such that 
	\begin{equation}\label{Hironaka}
		C_{nd} = \bigoplus_{i=1}^s q_i \mathbb{C}[p_1,\dots,p_k]
	\end{equation}
	as a $\mathbb{C}[p_1,\dots,p_k]$-module. The set of polynomials $\{q_1,\dots, q_s\}$ is called the secondary invariants and when combined with the homogeneous system of parameters, they form a set of fundamental invariants. The number of primary invariants is known to be $k=(d-1)n^2+1$ (see for instance an argument by Teranishi~\cite{T}). The decomposition itself has been largely unknown, even for $C_{42}$, where only the primary invariants were presented by Teranishi \cite{T}. Drensky and La~Scala~\cite{DLaS} provided some degree bounds for the secondary invariants of $C_{42}$, but due to an incomplete set of defining relations, the decomposition remained elusive. Now, with the acquisition of all defining relations of $C_{42}$, we are able to find all the secondary invariants presented in Appendix \ref{Secondary_Invariants} and thus, the quest of searching an answer of what is an invariant of pairs of $4\times 4$ matrices is completed in Subsection~\ref{Final}.

	As an application of our description, we reproved the results of~\cite{EGNT} on the defining relations of the invariant commuting variety of two $4 \times 4$ matrices and the fourth Calogero-Moser space. In fact, this can possibly be further extended to obtain more descriptions of other subvarieties embedded into the GIT quotient of pairs of matrices.
	
	Concerning larger matrix sizes, \DJ okovi\'{c}~\cite{Do} presented a minimal generating set consisting of 173 polynomials for $C_{52}$ and computed its Hilbert series. However, it is important to note that the form of these polynomials differs from those in~\cite{ADS} and~\cite{DS}, indicating that they might not be well-suited for expressing simple defining relations. Regarding $C_{62}$, the current knowledge is limited to the Hilbert series only (see~\cite{Do}). The methods developed in our work could potentially find application in resolving these open problems.
	
	\subsection*{Acknowledgments} We would like to thank Ivan Shestakov, Efim Zelmanov, Vyacheslav Futorny, Vladimir Dotsenko, Yura Berest and Xiaojun Chen for their valuable discussions.

	\section{Preliminaries}
	
	We denote by $\mathcal{M}_{n}$ the $\mathbb{C}$-algebra of $n\times n$ matrices over the field of complex numbers $\mathbb{C}$. Let $d$ be a positive integer and consider the action of the general linear group $\mathrm{GL}_n(\mathbb{C})$ on the direct product $\mathcal{M}_n^d$ of $d$ copies of $\mathcal{M}_n$ by simultaneous conjugation sending a $d$-tuple $(X_1,\dots ,X_d)$ to a $d$-tuple $(gX_1g^{-1},\dots,gX_dg^{-1})$ for any $g \in \mathrm{GL}_n(\mathbb{C})$. This action induces an action of $\mathrm{GL}_n$ on the algebra
	$\mathbb{C}[\mathcal{M}_n^d]$ of polynomial functions on $\mathcal{M}_n^d$, i.e.~for $\varphi \in \mathbb{C}[\mathcal{M}_n^d]$ and $g\in \mathrm{GL}_n$:
	\[
	(g\ast \varphi)(X_1,\dots,X_d)= \varphi(g^{-1}X_1g,\dots,g^{-1}X_dg).
	\]
	Note that one can identify $\M_n^d$ with the affine space $\mathbb{C}^{dn^2}$. In general, the description of the GIT-quotient $\mathcal{M}_n^d/\!\!/\mathrm{GL}_n$ is essentially an open problem of  invariant theory.

	Given that an algebraic variety $V$ is entirely characterized by its algebra of polynomial functions $\mathbb{C}[V]$, our objective is to describe the algebra $\mathbb{C}[\mathcal{M}_n^d/\!\!/\mathrm{GL}_n]$. However, by Mumford~\cite{Mu}, we can instead study the isomorphic algebra of $\mathrm{GL}_n$-invariant polynomials	
	\[ \mathbb{C}[\mathcal{M}_n^d]^{\mathrm{GL}_n}= \{ \varphi \in \mathbb{C}[\mathcal{M}_n^d] \mid g\ast \varphi = \varphi, \, \forall g\in \mathrm{GL}_n \}.
	\]
	Traditionally, the algebra of invariants $\mathbb{C}[\mathcal{M}_n^d]^{\mathrm{GL}_n}$ is denoted by $C_{nd}$. Conjecturing that invariants of matrices of arbitrary sizes are generated by finitely many trace functions, Artin~\cite{A} showed that this is indeed the case for invariants of small-sized matrices. Later, the conjecture was independently proven by Razmyslov~\cite{Ra}, Procesi~\cite{Pr}, Sibirski\u{\i}~\cite{Sib} and Helling~\cite{H}. 
	
	\begin{theorem}[\cite{Pr}] The algebra $C_{nd}$ is generated over $\mathbb{C}$ by all traces of products of generic matrices $\Tr(A_{i_1}A_{i_2}\cdots A_{i_j})$, with $j\leq 2^n-1$.
	\end{theorem}
	
	Another appropriate upper bound on the length of the words inside the trace function was given by Razmyslov~\cite{Ra} to be $n^2$. In fact, it is known that the exact upper bound is the same number determined by Dubnov-Ivanov-Nagata-Higman theorem in the context of PI-algebras. The algebra $C_{nd}$ is also referred to as pure or commutative trace algebra. For further background on these algebras we refer to Drensky and Formanek~\cite{DF}. The following result, often called the second fundamental theorem, affirms that the defining relations are derived from the Cayley-Hamilton theorem.
	
	\begin{theorem}[\cite{Ra}]
		All defining relations of the pure trace algebra $C_{nd}$ over a field of characteristic zero are consequences of the Cayley-Hamilton theorem.
	\end{theorem}
	
	%%%%%%%%%%%%%%%%%%%%%%%%%%%%%%%%%%%%%%%
	
	\subsection{Generators and defining relations} 
	Dubnov showed~\cite{D} that the only free algebra among pure trace algebras happens to be the following:
	\[
	C_{22}= \mathbb{C}[\Tr(X), \Tr(Y), \Tr(X^2), \Tr(XY), \Tr(Y^2)]. 
	\]
	Moreover, it is known that $C_{23}$ is a polynomial algebra with a single relation from the work of Sibirski\u{\i}~\cite{Sib} and Formanek~\cite{F1}. For $C_{2d}$ with $d>2$, some polynomial relations appear among the generators. With a choice of specific traceless generic matrices, Drensky~\cite{D2m} is able to give a better presentation of the defining relations. Another reason to change to traceless matrices it due to Procesi~\cite[Section 5]{Pr}, since there is an isomorphism of $\mathbb{Z}^d$-graded algebras:
	\[
	C_{nd} \cong \mathbb{C}[u_1,u_2,\dots,u_d] \otimes C_{n,d}(0),
	\]
	where $C_{n,d}(0)$ is the algebra of $\GL_n$-invariant polynomial functions on the direct product of $d$-tuples of the subspace of traceless $n\times n$ matrices. In fact, those $u_i$'s are the traces of generic matrices. 
	
	For $C_{32}$, Teranishi~\cite{T} proves that if the traces of the following ten words 
	\begin{equation}\label{hsop_C32}
		X, \, Y, \, X^2,\, XY, \, Y^2,\, X^3,\, X^2Y, \, XY^2,\, Y^3,\, X^2Y^2 
	\end{equation}
	are zero, then the value of the trace of any word in generic matrices is also zero. Applying Hilbert theorem, he argues that $C_{32}$ is integral over the polynomial subalgebra $A$ in those ten variables and he establishes its Hironaka decomposition~$C_{32}=A\oplus \Tr( XYX^2Y^2)A$. Clearly, a defining relation containing the square of~$\Tr(XYX^2Y^2)$ is eminent and it appears to be quite complicated in the work of Nakamoto~\cite{Na}. Once again, shifting to the traceless versions 
	\[ 
	A\coloneq X-\frac13\Tr(X)I_3, \qquad B\coloneq Y-\frac13\Tr(Y)I_3
	\] of the generic matrices $X$ and $Y$, along with replacing the last two generators with the traces of $[A,B]^2$ and $[A,B]^3$, where throughout the article $[A,B]=AB-BA$, appears to significantly simplify~\cite{ADS} the defining relation. 
	
	In~\cite[Proposition 5.2]{T} a set of 17 primary invariants were found by Teranishi for $C_{42}$. Using the fact that $\GL_2$ acts via the standard representation on the two-dimensional space spanned by two generic matrices, Drensky and Sadikova~\cite{DS} studied the induced $\GL_2$-action on $C_{42}$. They demonstrated the existence of a minimal set of generators that spans a semisimple graded $\GL_2$-submodule of $C_{42}$ and list the generators of each simple component.  We provide the set of generators arranged by degrees in Appendix \ref{generators_appendix}. The developed techniques to find the defining relations appear to be quite computationally complicated yielding only in the discovery of the low degree relations~\cite{DLaS}. Hoge~\cite{Hoge} adopted these methods to give a presentation of $C_{33}$.

	\subsection{Hilbert Series and Hironaka Decomposition}
	One of the useful tools in this study appears to be the usage of the Hilbert series---the formal power series~$\sum_{i,j\geq 0} \dim V_{i,j} \, t^i s^j$, where~$C_{n,d}=\oplus_{i,j\geq 0} V_{i,j}$ is considered as a bigraded vector space. The bigraded Hilbert series of $C_{n2}$ are known to be symmetric rational functions in two variables and for $n=4$ they were computed by Teranishi~\cite{T} and independently by Berele and Stembridge~\cite{BStem}. For our purposes, it suffices to consider the simply graded Hilbert series of $C_{42}$, that is obtained by equalizing the variables, 
	which in its irreducible form is mentioned in \cite[Table 3]{Do}:
	\begin{equation}\label{hilbert_C42}
		\dfrac{\left(1-t^2+t^4\right) \left(1-t-t^3+ t^4+2 t^5+ t^6-t^7-t^9+t^{10}\right)}{(1-t)^3 \left(1-t^2\right)^4 \left(1-t^3\right)^5 \left(1-t^4\right)^5}.
	\end{equation}
	
	Following \cite{Do}, let us rescale~\eqref{hilbert_C42} by $ \left( 1+{t}^{2} \right) \left( 1+t+{t}^{2} \right) \left( 1+{t}^{3} \right) ^{2} $ to obtain a fraction with the following new numerator
	\begin{eqnarray}\label{numerator}
		1+2\,{t}^{5}+2\,{t}^{6}+2\,{t}^{7}+4\,{t}^{8}+4\,{t}^{9}
		+4\,{t}^{10}+4\,{t}^{11}+2\,{t}^{12}+4\,{t}^{13} \\\nonumber
		+4\,{t}^{14}+4\,{t}^{15}+4\,{t}^{16}
		+2\,{t}^{17}+2\,{t}^{18}+2\,{t}^{19}+{t}^{24}, 
	\end{eqnarray}
	and denominator $\left( 1-t \right) ^{2} \left( 1-{t}^{2} \right) ^{3}
	\left( 1-{t}^{3} \right) ^{4} \left( 1-{t}^{4} \right) ^{6}
	\left( 1-{t}^{6} \right) ^{2}$. 
	This hints at a Hironaka decomposition for $C_{42}$ with 17 primary and 48 secondary invariants of their corresponding degrees. A set of 17 primary invariants was presented by Teranishi~\cite{T} using similar methods he employed for $C_{32}$. Following the beneficial practice of the usage of traceless versions, we choose the primary invariants to be the traces of the following:
	\begin{eqnarray}\label{primary_invariants}
		\nonumber& X, Y \\ 
		\nonumber& A^2, AB, B^2 \\
		& A^3, A^2B, AB^2, B^3\\
		\nonumber& A^4, A^3B, A^2B^2, AB^3, B^4\\
		\nonumber& \frac12[A,B]^2, [A,B]^2A^2, [A,B]^2B^2.
	\end{eqnarray}
	Only some bounds on degrees of the secondary invariants were suggested in~\cite{DLaS} as a consequence of obtaining low degree defining relations. In Subsection~\ref{Final} we explain how we obtain the set of secondary invariants arranged by their degrees in Appendix \ref{Secondary_Invariants}, matching the numerator~\eqref{numerator}.

	\subsection{Invariant Commuting Variety and Calogero-Moser spaces} 
	Let us consider the following subspace of $\M_n\times \M_n$:
	\begin{equation*}
		\{ (X,Y)\in \M_n\times \M_n 
		\mid \mathrm{rank}([X,Y] +I_n)=1\}\, .
	\end{equation*}
	The GIT quotient of this space with respect to the $\GL_n$-action is called the $n$-th \textit{Calogero-Moser space}, and it is denoted by $\mathcal{C}_n$. 
	The Calogero-Moser spaces were studied in detail by Wilson ~\cite{W} and he proved that $\mathcal{C}_n$ is a smooth affine irreducible complex symplectic variety of dimension $2n$. These spaces appear to be of substantial importance in many areas, such as algebraic geometry (Hilbert schemes), deformation theory (symplectic reflection algebras), representation theory (double affine Hecke algebras) and Poisson geometry. Consequently, there is a strong motivation to obtain a clear and explicit description of their coordinate rings. Since the coordinate ring of $\mathcal{C}_n$ is a quotient algebra of $C_{n2}$, it is of great interest to see the refinement of the fundamental theorems on these spaces. 
	
	The commuting variety $\{ (X,Y)\in \M_n\times \M_n	\mid [X,Y]=0\}$ while being interesting on its own, its GIT quotient called the $n$-th \textit{invariant commuting variety}, which we denote by ${\mathcal Com}_n$. It is known that $\CC[{\mathcal Com}_n]$ is isomorphic to $(\mathbb {C}[x,y]^{\otimes n} )^{S_{n}}$ (see~\cite{Dom} and~\cite{Va}). Calogero-Moser spaces and invariant commuting varieties appear to share many similarities. It is known that 
	\begin{eqnarray*}
		\mathbb{C}[\mathcal{C}_2] &\cong& \mathbb{C}[a_1,a_2,a_3,a_4,a_5]/(a_4^2-a_3a_5-1),\\ [0.2cm]
		\mathbb{C}[{\mathcal Com}_2] &\cong& \mathbb{C}[a_1,a_2,a_3,a_4,a_5]/(a_4^2-a_3a_5) 
	\end{eqnarray*}
	and for $n=3$ the full description is obtained in \cite{NT-TG}:
	\[
	\CC[\mathcal{C}_3] \cong \CC[a_1, \dots ,a_9]/ I_1 \ , \qquad \mathbb{C}[{\mathcal Com}_3] \cong \CC[a_1, \dots ,a_9]/ I_0,
	\]
	where $I_{\delta}$ is generated by the following five relations:
	\begin{align*}
		r_{1}&=a_3a_9-2a_4a_8+a_5a_7,\\
		r_{2}&=a_5a_6-2a_4a_7+a_3a_8,\\
		r_{3}&=9\delta a_3-a_3a_4^2+a_3^2a_5+6a_6a_8-6a_7^2,\\
		r_{4}&=9\delta a_4-a_4^3+a_3a_4a_5+3a_6a_9-3a_7a_8,\\
		r_{5}&=9\delta a_5-a_4^2a_5+a_3a_5^2+6a_7a_9-6a_8^2.
	\end{align*}
	We refer to Appendix \ref{generators_appendix} for the meaning of each of these generators.
	
	Strikingly, the complicated defining relation of $C_{32}$ with the help of the defining relations of $\CC[\mathcal{C}_3]$ can be rewritten easily (see~\cite{NT-TG}):
	\[a_{21}^{2}+\frac{4a_{15}^{3}}{27}-\frac1{27}(r_{3}r_{5}-r_{4}^{2})-\frac1{18} (a_{3}{r_{1}^{2}}-2a_{4}r_{1}r_{2}+a_{5}{r_{2}^{2}})=0.
	\]
	For $n=4$ the generators of $\CC[\mathcal{C}_4]$ and $\mathbb{C}[{\mathcal Com}_4]$ are exactly $a_1,\dots, a_{14}$, since the other generators of $C_{42}$ contain commutators inside their trace function. The ideal of the defining relations for the coordinate ring of the fourth Calogero-Moser space appears to be generated by 12 relations, while for the invariant commuting variety there are 15 relations \cite{EGNT}. We will see as an application of knowing the full description of $C_{42}$, how one can deduce the presentation of $\CC[\mathcal{C}_4]$ and $\mathbb{C}[{\mathcal Com}_4]$ in Section~\ref{applications}.

	\subsection{\texorpdfstring{Poisson algebra structure on $C_{n2}$}{Poisson algebra structure on Cn2}}
	Let us consider the symplectic form 
	\[
	\omega ((X_1,X_2) \, , \, (Y_1,Y_2))=\textrm{Tr}(X_1Y_2-X_2Y_1).
	\]
	on the pairs of matrices $\mathcal{M}_n\times \mathcal{M}_n\cong T^*(\mathcal{M}_n)$. This form endows $\mathbb{C}[\mathcal{M}_n\times \mathcal{M}_n]$ with a Poisson algebra structure by
	\[
	\{f,g\} = \omega (X_f,X_g), \quad f,g\in \mathbb{C}[\mathcal{M}_n\times \mathcal{M}_n].
	\]
	Since the symplectic form is $\GL_n$-invariant, this induces the algebra $C_{n2}$ with a Poisson algebra structure. Nevertheless, a more robust approach is needed to ascertain the value of the Poisson bracket on specific generators. For that, let us consider the two-dimensional vector space $V=\mathbb{C} x\oplus\mathbb{C} y$ and define the following symplectic form on it 	
	\[
	\omega(x,y)=-\omega(y,x)=1, \qquad \omega(x,x)=\omega(y,y)=0.
	\]
	There is a Leibniz algebra structure~\cite{Ko} on the free noncommutative $\mathbb{C}$-algebra $R$ generated by two elements $x$ and $y$, defined as follows:
	\begin{equation}
		\{u_1\cdots u_p, v_1\cdots v_q\}=\sum\limits_{\substack {
				1\leq i \leq p \\ 	  1\leq j \leq q	}}\omega(u_i,v_j)u_{i+1}\cdots u_pu_1\cdots u_{i-1}v_{j+1}\cdots v_q v_1\cdots v_{j-1},
	\end{equation}
	where the elements $u_1,\dots,u_p,v_1,\dots, v_q$ are either $x$ or $y$. For the definition and further details regarding Leibniz algebras see~\cite{Lo}. Now, we define the subspace~$\textrm{Com}(R)=\textrm{Span}\{ab-ba \mid a,b\in R\}$ and note that it belongs to the center of $R$. Hence, $\textrm{Com}(R)$ is a central ideal of $R$ and one can consider the quotient Leibniz algebra $\N=R/\text{Com}(R)$. Furthermore, this quotient appears to be a Lie algebra~\cite{G}. This Lie algebra is a particular example of a Necklace Lie algebra, and for a comprehensive combinatorial definition employing quivers, we refer the reader to~\cite{BLB}. It holds a central position at the core of our work, as it facilitates the computational aspect through the application of the Lie algebra structure to the invariants of pairs of matrices. The trace map defined as follows (see~\cite{G})
	\begin{equation}\nonumber
		\begin{split}
			tr\colon \N&\rightarrow C_{n2}\\
			x^{k_1}y^{l_1}\cdots x^{k_m}y^{l_m}&\mapsto \big( (X,Y)\mapsto \mathrm{Tr}(X^{k_1}Y^{l_1}\cdots X^{k_m}Y^{l_m})\big)
		\end{split}
	\end{equation}
	is a well-defined Lie algebra homomorphism, which equips us with tools to compute the brackets on $C_{n2}$, turning the later to a Poisson algebra. Using this approach, an alternative way to establish the defining relation of $C_{n2}$ for $n=3$ was offered in~\cite{GNT}.

	\subsection{Necklaces} A necklace in $\N$ is formed by equivalence classes of cyclic permutations of words in a~$2$-letter alphabet and, as a standard consideration, we shall use as representative the biggest one with respect to the alphabetical order.
	The usual degree lexicographical order in the space of two letter words induces a \emph{degree lexicographical order} in the space of necklaces by comparing the representative word of their respective equivalence class.
	
	\begin{definition}
		A necklace is said to be $CH_n$ if at least one of the words in its equivalence class contains the same subword $n$ times consecutively. 
	\end{definition}
	
	\begin{definition}
		For a fixed bidegree~$(r, s)$, we say that a pair of necklaces $(w_1, w_2)$ is a \emph{breaking pair} if it satisfies the following properties:
		\begin{itemize}
			\item The sum of their bidegrees is exactly~$(r + 1, s + 1)$,
			\item The first one has at least degree~$2$ and it is smaller with respect the degree lexicographical order,
			\item If the degree of the first necklace is~$2$, then the second necklace is $CH_n$,
			\item The trace of at least one of them is not a term in the linear expansion of the minimal set of generators.
		\end{itemize}	
	\end{definition}
	These pairs will be actively used to compute the value of the algebraic expressions of the traces of necklaces that are not $CH_4$. 
	For instance, for $n=4$ in bidegree~$(3, 2)$ the pair $(B^2, A^4B)$ is a breaking pair, but the pair $(A^2B, ABAB)$ is not, since~$\Tr(A^2B)$ is exactly $a_7$ and $\Tr(ABAB)$ appears in the linear expansion of~$a_{15}$.
	
	\section{Description of the algorithmic approach}

	\subsection{Overview}\label{ss:overview}
	Let us sketch the general idea of our procedure.
	First, we shall inductively obtain the unique description of the trace values of all necklaces of degree less or equal than~$11$ in terms of the generators from Appendix ~\ref{generators_appendix}. We shall do this using the Cayley-Hamilton theorem for $CH_4$ when it is possible, and breaking pairs when it is not. With this information we shall obtain the complete description of the Poisson adjoint maps $\{a_5, - \}$ and $\{a_6, -\}$ on the generators $a_3,\dots,a_{32}$. 
	
	Then, starting in degree~$12$, the approach above will give us the description of each trace of necklace modulo relations. Whenever a new relation emerges, it is incorporated to the ideal of relations together with its Poisson bracket with~$a_5$ and~$a_6$. Using the already known Hilbert series, we can stop the computation when the full dimension is reached and shift to the next step.
	
	In this way, we can obtain almost the whole ideal of relations. The last piece of the puzzle will be unraveled by a careful exploration of the central elements of the Necklace Lie algebra.
	
	\subsection{Initial considerations}\label{subs:order}
	
	The ambivalence between $A$ and $B$ will allow us to work only in bidegrees $(r, s)$ with $r \geq s$. This inner symmetry will appear in all aspects of our computations, including the expressions of traces of necklaces, relations and Poisson brackets. 
	
	While finding the value of the bracket of two necklaces in terms of generic matrices is very fast, when they are replaced by their traceless versions this computation can be far more costly by the big amount of substitutions that have to be made. Because of this, it is recommended that anytime a bracket is computed its value should be stored in the memory, since it will probably be used again in a higher degree computation.

	\subsection{First non-trivial degree}
	
	To demonstrate our approach, we shall commence in degree~$n=5$, since up to degree $4$ all the traces are simply linear combinations of the generators.
	
	We start in bidegree~$(5, 0)$, where we only find one expression,~$\Tr(A^5)$, which can be decomposed using the Cayley-Hamilton theorem and it is equal to~$\frac{5}{6}a_{3}a_{6}$. Similarly, in bidegree~$(4, 1)$ we get
	\begin{align*}
		\Tr(A^4B) &= \dfrac{1}{2} \Tr(A^2) \Tr(A^2B) + \dfrac{1}{3} \Tr(A^3) \Tr(AB) \\
		&= \dfrac{1}{2}a_3 a_7 + \dfrac{1}{3}a_4 a_6.
	\end{align*}
	
	In bidegree~$(3, 2)$ things get more interesting. There are two traces, $\Tr(A^3B^2)$ and~$\Tr(A^2BAB)$ that are not $CH_4$. In this bidegree, there is only one breaking pair:~$(B^2, A^4B)$. On one hand, the Poisson bracket of the traces of this pair is:
	\[
	\{\Tr(B^2), \Tr(A^4B)\} = -4 \Tr(A^3B^2) -4 \Tr(A^2BAB). 
	\]
	On the other hand,
	\begin{align*}
		\{\Tr(B^2), \Tr(A^4B)\} &= \{ \Tr(B^2), \dfrac{1}{2} \Tr(A^2) \Tr(A^2B) + \dfrac{1}{3} \Tr(A^3) \Tr(AB) \} \\
		&= 4 \Tr(AB) \Tr(AB^2) + 2 \Tr(A^2B) \Tr(B^2) + \dfrac{2}{3} \Tr(A^2) \Tr(B^3).
	\end{align*}
	
	Equalizing both equations, together with the actual definition of the generator
	\[
	a_{16} = \Tr([A, B]^2A) = -\Tr(A^3B^2) + \Tr(A^2BAB),
	\]
	we obtain that:
	\begin{align*}
		\Tr(A^3B^2) &= \frac{1}{12} (a_5 a_6+6 a_4 a_7+3 a_3 a_8-6 a_{16}), \\
		\Tr(A^2BAB) &=\frac{1}{12} (a_5 a_6+6 a_4 a_7+3 a_3 a_8+6 a_{16}).
	\end{align*}
	
	This completes the case $n=5$, since the remaining expressions can be immediately found by making use of the inner symmetry.
	
	\subsection{Degrees less than 12}\label{subs:lower}
	
	It is known that relations appear starting from degree~$12$, so $\Tr(w)$ has a unique expression in terms of the generators for each $w$ of degree less or equal than $11$. Our task will be to efficiently compute these expressions.
	Let us proceed inductively by degree, denoted by $n$.
	An immediate application of the Cayley-Hamilton theorem to the traces of the necklaces that are~$CH_4$ reduces them to an algebraic expression involving only traces of lower degree necklaces. Therefore, it will give us their unique description in terms of the generators.
	
	Then, to find the expressions of the necklaces that are not~$CH_4$, we shall proceed by bidegree~$(r, s)$ starting with the highest~$r$ possible such that~$r + s = n$ where a non-$CH_4$ necklace appears.
	
	Let us consider all possible breaking pairs of bidegree $(r, s)$. For each such pair~$(w_1, w_2)$, we directly compute the bracket $\{ \Tr(w_1), \Tr(w_2) \}$. This results in a linear combination of traces of necklaces of bidegree $(r, s)$, in which the traces of~$CH_4$ necklaces are further substituted by their corresponding algebraic expression on the generators.
	
	The properties of breaking pairs guarantee us that at least one of the $\Tr(w_i)$ can be substituted by an algebraic expression of lower terms, thus, by the generators. Then, by making these substitution before computing the bracket $\{ \Tr(w_1), \Tr(w_2) \}$, we obtain another expression just in lower degree terms.
	
	Equalizing both expressions, we obtain a linear equation where the variables are traces of non-$CH_4$ necklaces. Doing this for all breaking pairs results in a linear system of equations. Furthermore, if there is a generator $a_i$ in this bidegree, we add the equation formed by its expansion as a linear combination of traces of necklaces.
	The solution of this completely determined system will give us the unique algebraic expression of all traces of non-$CH_4$ necklaces in terms of the generators. Then, we continue to~$(r-1, s+1)$ recursively.
	
	Note that to efficiently compute these expressions it is not necessary to calculate the linear equations for all breaking pairs. We can randomly choose one pair at a time and add the resulting corresponding equation one by one, until we reach a completely determined system.

	\subsection{Higher degrees}
	
	Starting from degree $12$ we know that relations will start to occur. Consequently, the expressions of traces of necklaces will stop being unique, since they are modded by this ideal of relations. Whenever we have to save these values, it does not matters which representative we choose. 
	We will inductively construct the ideal $I$ by gathering the relations that will appear in each step.
	
	The same process described above will now lead to an inconsistent system of equations. It is an easy linear algebra exercise to find the minimal set of polynomials on the generators that will be incorporated to $I$ to make the system consistent. There is a special case to be considered and it is bidegree $(n/2, n/2)$ for even $n$. In this case, all the breaking pairs are not enough to generate all the relations, whilst exactly one might be missing. The way to compute this extra needed equation comes from a careful exploration of the center of the Necklace Lie algebra. Let us consider the expression~$\Tr([A, B]^{n/2})$. On one hand, we can expand it and substitute the traces of the $CH_4$ necklaces. On the other hand, we can use the Cayley-Hamilton theorem to break $[A, B]^{n/2}$ into smaller degree terms, whose expressions on generators are already known. The equalization of both expansions will be incorporated to the system of equations.
	
	After obtaining all the relations in degree~$n$, we compute the Hilbert series of $I$ and compare it with~\eqref{hilbert_C42}. If they do not coincide, we move to the next degree.

	\subsection{Refined execution}\label{ss:refined}
	
	Although the algorithmic approach described above is valid, evidently an optimization is necessary. Calculating explicit Hilbert series or computing Poisson brackets at higher degrees can be very demanding. For this reason, we found some shortcuts using the Poisson structure and general theory of Gröbner bases. 
	
	We still need to compute the expressions of the traces of necklaces, but we found experimentally that this will only be needed until degree~$16$. Let us explain how to further optimize this process.
	
	Anytime that we update the ideal, we compute a reduced graded Gröbner basis of it, so the ideal membership problem or the computation of Hilbert series is fast.
	
	Therefore, after degree~$12$, at each bidegree $(r, s)$ we know how many dimensions we are missing by an examination of the Hilbert series of our current ideal $I$. If this number is greater than zero, we try to find new independent elements coming from the Poisson bracket of~$a_5$ with the relations of bidegree~$(r + 1, s - 1)$. If they are not enough, we try to find new polynomials coming from the Poisson bracket of~$a_6$ with the relations of bidegree~$(r - 2, s + 1)$. 
	
	Now, to compute the expressions of the traces of non-$CH_4$ necklaces we can proceed in the following way. We select randomly one breaking pair at a time and we consider the resulting equation from the procedure described before. Depending on its situation with the previously obtained equations we do the following
	\begin{itemize}
		\item If it is linearly dependent, we discard it.
		\item If it is linearly independent and does not make the system incompatible, we add it.
		\item If it makes the system incompatible, we compute the relation in terms of the generators, we incorporate it to $I$ if it does not belong to it yet, and then we discard it.
	\end{itemize}
	When the system of equations reaches full rank and $I$ has full dimension, we stop the procedure. Note that if we are in bidegree $(n/2, n/2)$ for even $n$, the first equation to consider should be the special one coming from $\Tr([A, B]^{n/2})$.
	
	When we have all the relations up to degree~$16$, we found that we do not need to compute the expressions of the traces of necklaces any more. The consecutive application of the Poisson bracket of $a_5$ or $a_6$ to the already obtained relations fills up the dimensions of our graded ideal. This releases us from making a huge amount of computations until~$20$, degree where the last new relation appears. 
	
	\subsection{Outcome}
	
	After applying the procedure explained above, we obtained the following data:
	
	The polynomial ideal $I$ is generated by the 105 relations listed in the ancillary file. The maximum degree found in this set of generators is~$(10, 10)$.
	Furthermore, only 8 relations (also listed in the ancillary file) are needed to generate~$I$ under the repeated adjoint maps~$\{ a_5, - \}$ and~$\{ a_6, - \}$. More precisely, three of them appear in bidegrees: 
	\[
	(6, 6), \quad (7, 7), \quad (8, 8),
	\]
	directly from the expansion of the respective power of the commutator. The other five are in bidegrees:
	\[
	(7, 5), \quad (6, 6), \quad (7, 6), \quad (8, 6), \quad (8, 7) .
	\]
	
	Therefore, we can state the main theorem:
	\begin{theorem}\label{th:main}
		There is an isomorphism
		\[
		\CCC \cong \CC[a_1, a_2] \otimes \dfrac{\CC[a_3, \dots, a_{32}]}{I}
		\]
	\end{theorem}
	\begin{proof}
		With the help of any computer algebra system it is an easy computation to check that the 105 polynomials generating~$I$ are relations in $\CCC$. Then, checking that the Hilbert series of $I$ coincides with~\eqref{hilbert_C42} proves the theorem.
	\end{proof}

	\subsection{Final considerations}
	
	The method discussed above is a detailed explanation on how we obtained the full set of relations generating the ideal. Since the validity of the main theorem can be checked without examining following the detailed procedure~\ref{ss:overview}--\ref{ss:refined}, we consider that it is cleaner and clearer to present it in this way.
	
	Since we are working over a polynomial ring of~$30$ variables with an ideal generated by~$105$ polynomials---some of them huge---the computation of Gröbner bases can be a really demanding task even for modern computers. The only way we were able to calculate the Hilbert series to conclude the proof of Theorem~\ref{th:main} was using the Hilbert driven Gröbner basis algorithm implemented in \texttt{Macaulay2}~\cite{Mac2}, with the grading induced by the sum of the lengths. The only monomial order we could find for this to be executed in a reasonable amount of time is the following:
	Let~$a_i, a_j$ be of bidegree $(r_i, s_i)$ and $(r_j, s_j)$, respectively. Then $a_i < a_j$ whenever $(r_i, s_i)$ is lower than $(r_j, s_j)$ with respect to the lexicographical order. There are exactly three cases when we have two generators of the same bidegree, choosing~$a_{15} < a_{12}$, $a_{21} < a_{19}$ and $a_{27} < a_{25}$. 
	
	The main reason for the high efficiency of this method is the ability to transition from lower to higher degree relations via the Poisson structure. Previous approaches were limited to generating relations within the same degree, and having to again fully decompose the problem when moving to a higher total grading. Since the highest degree required is 20, the decomposition should be carried out up to that level. However, with our method, the computation can stop at degree 16, the highest among our 8 generating relations. In fact, not even the entire degree 16 needs to be computed, as only one of the extra relations is missing, therefore, a partial analysis of bidegree (8, 8) suffices.
	
	\subsection{\texorpdfstring{Hironaka decomposition for $C_{42}$}{Hironaka decomposition for C42}}\label{Final} 
	After getting the full ideal of relations, to get the Hironaka decomposition of $C_{42}$ becomes a much easier task. After an examination of the modified denominator~\eqref{numerator} of the Hilbert series, we choose as primary invariants the~$17$ traces~\eqref{primary_invariants} denoted by $a_1,\dots, a_{15},a_{18},a_{20}$ in Appendix~\ref{generators_appendix}. Then, we compute a Gröbner basis with respect to the graded reverse lexicographical order of the ideal generated by $a_1,\dots, a_{15},a_{18},a_{20}$ together with the~$105$ relations obtained in Theorem~\ref{th:main}. Finally, it is a matter of checking which are the 48 normal monomials in this Gröbner basis.
	\begin{corollary}
		The algebra $C_{42}$ is a free $\mathbb{C}[a_1,\dots, a_{15},a_{18},a_{20}]$-module of rank~$48$, and a basis is listed in Appendix~\ref{Secondary_Invariants}.
	\end{corollary}
	
	%%%%%%%%%%%%%%%%%%%%%%%%%%%%%	
	\section{Applications to the invariant commuting variety and the Calogero-Moser space}\label{applications}

	Since the invariant commuting variety is a subvariety of the variety of pairs of matrices with respect to conjugation, there is a surjective algebra homomorphism from $C_{n2}$ to $\mathbb{C}[\Com_n]$, for arbitrary $n>1$.
	Therefore,  we can attempt to leverage the structure of $C_{n2}$ to obtain the complete description of $\CC[\Com_n]$. 
	
	If $C_{n2}$ and $\CC[\Com_n]$ are isomorphic to the polynomial ideals $\CC[a_1, \dots, a_k]/I$ and~$\CC[a_1, \dots, a_s]/J$, respectively, where the $a_i$ represent the trace of some expression, there is a well-defined surjection $\CC[a_1, \dots, a_k] \to \CC[a_1, \dots, a_s]$ sending any generator to its corresponding expression in the commuting variety. In other words, we obtain the following diagram:
	\[
	\xymatrix{
		0 \ar[d] & &  0 \ar[d] \\
		I \ar[d] \ar[rr] & & J \ar[d] \\
		\CC[a_1, \dots, a_k] \ar[d] \ar@{>>}[rr] & & \CC[a_1, \dots, a_s] \ar[d] \\
		\dfrac{\CC[a_1, \dots, a_k]}{I} \ar[d] \ar@{>>}[rr] & & \dfrac{\CC[a_1, \dots, a_s]}{J} \ar[d] \\
		0  & & 0 
	}
	\]
	
	Note that all arrows are algebra homomorphisms, and furthermore, the lower horizontal arrow is also a Poisson algebra homomorphism.
	
	If we know the description of $I$ and its induced morphism to $J$ happens to be a surjection, we will obtain the full description of $J$.
	Let us examine what happens in the known cases:

	For $n = 2$, it is established that $k = s = 5$, hence the surjection in the middle is an isomorphism. Nevertheless, $I = 0$ but $J = \langle a_4^2 - a_3a_5 \rangle$,  indicating that the induced morphism between the ideals of relations can never be a surjection.
	
	When $n = 3$, it is known that $k = 11$ and $s = 9$ and the morphism in the middle sends each generator to itself, except the last two that are mapped to zero. In this case, $I$ is the ideal generated by a single polynomial, while $J$ is an ideal generated by five polynomial relations~\cite{NT-TG}. Therefore, it is obvious that, again, the induced morphism is not a surjection.
	
	However, for $n = 4$  the number of relations increases significantly, so the situation may differ.
	It is known that $k=32$ while $s=14$  and, moreover, the polynomial ideal~$I$ is generated by $105$ equations while the ideal $J$ is generated by just $15$ (see~\cite{EGNT}). Looking at the form of the generators in Appendix~\ref{generators_appendix}, it is easy to guess the map---$a_i$'s are mapped to themselves for all $1\leq i \leq 14$ and to zero otherwise. Given that the rank of the Hironaka decomposition of $\CC[\Com_{4}]$ is already established~\cite{EG}, it is a straightforward exercise to verify that this map induces a surjection from $I$ to $J$, recovering the description of $J$ obtained in~\cite{EGNT}.

	In the case of Calogero-Moser spaces, things work in a similar fashion. We can construct a diagram as before, with the only variation being the image of the generators.
	
	For $n = 2$, the map is defined identically, but $J = \langle a_{4}^2 - a_3a_5 +1 \rangle$, so again, it is not possible to recover $J$ by this map.
	
	When $n = 3$, it is known that $s = 9$ and the map is defined as follows: the first nine generators are mapped to themselves while the last two go to the scalars $-3$ and $2$, respectively~\cite{NT-TG}. Once more, $I$ is generated by just one relation, whereas $J$ is generated by five of them. Consequently, $J$ cannot be recovered by this map.
	
	In the case of $n = 4$, the situation is analogous to that in the invariant commuting variety; however, determining the morphism is more intricate. Motivated by the identities that emerge when the pair of matrices is restricted to Calogero-Moser space $\C_4$, it can be proved by similar methods of~\cite{NT-TG} that the generators are mapped as follows:
	\begin{align*}
		&a_{i} \mapsto a_i, \quad \quad 1\leq i \leq 14,\\
		&a_{15}\mapsto 6, \quad \quad  a_{21}\mapsto 24, \quad a_{27}\mapsto 42, \quad a_{32}\mapsto 168,\\
		&a_{18}\mapsto 3a_3,  \quad a_{19}\mapsto 6a_4, \quad a_{20}\mapsto 3a_5, \\
		&a_{24}\mapsto 6a_3, \quad a_{25}\mapsto 6a_4, \quad  a_{26}\mapsto 6a_5,\\	
		&a_{28}\mapsto 6a_6, \quad a_{29}\mapsto 6a_7, \quad  a_{30}\mapsto 6a_8, \quad a_{31}\mapsto 6a_9	
	\end{align*}
	and the unmentioned generators are mapped to zero. Analogously, in this case, the induced morphism from $I$ to $J$ can be easily checked that it is a surjection, making use of the rank of the Hironaka decomposition of $\CC[\C_4]$, recovering the main result of~\cite{EGNT}.

	%%%%%%%%%%%%%%%%%% APPENDIX %%%%%%%%%

	\appendix
	\section{Generators} \label{generators_appendix}

	\renewcommand{\arraystretch}{1.5}
	
	\begin{center}
		\begin{tabular}{cl}
			\toprule
			\multicolumn{1}{l}{Degree} & Generators                                                       \\ \midrule
			1                          & $a_1=\Tr(X), \ a_2=\Tr(Y)$                                       \\ \midrule[0.1mm]
			2                          & $a_3=\Tr(A^2),\ a_4=\Tr(AB), \ a_5=\Tr(B^2)$                     \\ \midrule[0.1mm]
			3                          & $a_6=\Tr(A^3), \ a_7=\Tr(A^2B), \ a_8=\Tr(AB^2), \ a_9=\Tr(B^3)$ \\ \midrule[0.1mm]
			\multirow{3}{*}{4}         & $a_{10}=\Tr(A^4), \ a_{11}=\Tr(A^3B), \ a_{12}=\Tr(A^2B^2),$     \\  
			& $a_{13}=\Tr(AB^3
			), \ a_{14}=\Tr(B^4)$                            \\  
			& $a_{15}=\frac{1}{2}\Tr([A,B]^2)$                                 \\ \midrule[0.1mm]
			5                          & $a_{16}=\Tr([A,B]^2A), \ a_{17}=\Tr([A,B]^2B)$                   \\ \midrule[0.1mm]
			\multirow{4}{*}{6}         & $a_{18}=\Tr([A,B]^2A^2),$                                        \\  
			& $a_{19}=\Tr([A,B]^2(AB+BA)),$                                    \\  
			& $a_{20}=\Tr([A,B]^2B^2),$                                        \\  
			& $a_{21}=\frac13\Tr([A,B]^3)$                                     \\ \midrule[0.1mm]
			7                          & $a_{22}=\Tr([A,B]^3A), \ a_{23}=\Tr([A,B]^3B)$                   \\ \midrule[0.1mm]
			\multirow{4}{*}{8}         & $a_{24}=\Tr([A,B]^3A^2),$                                        \\  
			& $a_{25}=\frac12\Tr([A,B]^3(AB+BA)),$                             \\  
			& $a_{26}=\Tr([A,B]^3B^2),$                                        \\  
			& $a_{27}=\frac12\Tr([A,B]^4)$                                     \\ \midrule[0.1mm]
			\multirow{4}{*}{9}         & $a_{28}=\Tr([A,B]^3A^3),$                                        \\  
			& $a_{29}=\frac13\Tr([A,B]^3(A^2B+ABA+BA^2)),$                     \\  
			& $a_{30}=\frac13\Tr([A,B]^3(AB^2+BAB+B^2A)),$                      \\  
			& $a_{31}=\Tr([A,B]^3B^3),$                                        \\ \midrule[0.1mm]
			10                         & $a_{32}=\Tr([A,B]^3(A^2B^2-AB^2A-BA^2B+B^2A^2))$                 \\ \bottomrule
		\end{tabular}
		
	\end{center}

	\section{Secondary invariants} \label{Secondary_Invariants}

	\begin{center}
		\begin{tabular}{cl}
			\toprule
			Degree & Secondary Generators \\ \midrule
			0      &    1 \\ \midrule[0.1mm]
			5      &   $a_{16}$, $a_{17}$  \\ \midrule[0.1mm]   
			6      &   $a_{19}$, $a_{21}$ \\ \midrule[0.1mm]
			7      &   $a_{22}$, $a_{23}$ \\ \midrule[0.1mm]
			8      &  $a_{24}$, $a_{25}$, $a_{26}$, $a_{27}$  \\ \midrule[0.1mm]
			9      &  $a_{28}$, $a_{29}$, $a_{30}$, $a_{31}$  \\ \midrule[0.1mm]
			10      &  $a_{32}$, $a_{16}^2$, $a_{16} a_{17}$, $a_{17}^2$  \\ \midrule[0.1mm]
			11      &   $a_{16} a_{19}$, $a_{16} a_{21}$, $a_{17} a_{19}$, $a_{17} a_{21}$ \\ \midrule[0.1mm]
			12      &  $a_{19} a_{21}$, $a_{21}^2$  \\ \midrule[0.1mm]
			13      &  $a_{19} a_{22}$, $a_{19} a_{23}$, $a_{21} a_{22}$, $a_{21} a_{23}$  \\ \midrule[0.1mm]
			14      &   $a_{19} a_{24}$, $a_{19} a_{26}$, $a_{21} a_{27}$, $a_{22} a_{23}$  \\ \midrule[0.1mm]
			15      &  $a_{22} a_{24}$, $a_{23} a_{24}$, $a_{23} a_{25}$, $a_{23} a_{26}$  \\ \midrule[0.1mm]
			16      &  $a_{24} a_{25}$, $a_{25} a_{26}$, $a_{25} a_{27}$, $a_{27}^2$  \\ \midrule[0.1mm]
			17      &  $a_{26} a_{28}$, $a_{26} a_{29}$  \\ \midrule[0.1mm]
			18      &  $a_{27} a_{32}$, $a_{29} a_{30}$  \\ \midrule[0.1mm]
			19      &  $a_{21}^2 a_{22}$, $a_{21}^2 a_{23}$  \\ \midrule[0.1mm]
			24      & $a_{25} a_{27}^2$                \\ \bottomrule
		\end{tabular}
		
	\end{center}

	%%%%%%%%%%%%%%%%%%%%%%%%%%%%%%%%%%
	%\bibliography{bib_GNT2}
	%\bibliographystyle{amsplain-nodash}
	
	\providecommand{\bysame}{\leavevmode\hbox to3em{\hrulefill}\thinspace}
	\providecommand{\MR}{\relax\ifhmode\unskip\space\fi MR }
	\providecommand{\MRhref}[2]{%
		\href{http://www.ams.org/mathscinet-getitem?mr=#1}{#2}
	}
	\providecommand{\href}[2]{#2}

\end{document}